\documentclass[letterpaper]{amsart}
\usepackage[margin=1in]{geometry} 
\usepackage[utf8]{inputenc}
\usepackage{amsmath,amsthm,amssymb,amsfonts,mathrsfs}
\usepackage{dsfont}
\usepackage{bookmark}
\usepackage{xcolor}
\usepackage[hyperpageref]{backref}

\renewcommand*{\backref}[1]{}

\renewcommand*{\backrefalt}[4]{%
  \ifcase #1
    (No citations)
  \else
    \space #2%
  \fi
}

\newcommand{\FR}{\mathrm{FR}}

\numberwithin{equation}{section}
    
\newtheorem{theorem}{Theorem}[section]

\newtheorem{proposition}[theorem]{Proposition}
\newtheorem{corollary}[theorem]{Corollary}

\theoremstyle{definition}

\newtheorem{question}[theorem]{Question}

\theoremstyle{remark}
\newtheorem{remark}[theorem]{Remark}

\title[Uncertainty principles]{Uncertainty principles and singular potentials}
\author[A. Iosevich]{Alex Iosevich}
\author[C. Park]{Chamsol Park}
\address{Department of Mathematics, University of Rochester, Rochester, NY 14627, USA}
\email{alex.iosevich@rochester.edu}
\address{Department of Mathematics, University of Rochester, Rochester, NY 14627, USA}
\email{cpark63@ur.rochester.edu}

\thanks{The first-listed author was supported by NSF DMS-2506858.}
\date{}

\usepackage{pgfplots}
\pgfplotsset{compat=1.16}

\begin{document}

\begin{abstract}
We establish uncertainty principles on compact Riemannian manifolds without boundary in the setting of Laplace-Beltrami operators, including the case of real-valued singular potentials. We replace the classical homogeneity assumption by a quantitative spectral condition and obtain corresponding stability versions of uncertainty inequalities. In particular, we prove that
\[
(1-\epsilon-\epsilon')^2 \leq \frac{|E|}{|M|}\cdot \# X_S \cdot \sup_{x\in E} \frac{A_S(x)}{\frac{\# X_S}{|M|}},
\]
which recovers the classical bound in the homogeneous case, quantifies its deterioration in the presence of spectral inhomogeneity, and is shown to be sharp in general. In {\it dimension one}, we show that the homogeneity condition holds automatically, and we complement this rigidity by incorporating Fourier-ratio complexity bounds, yielding a quantitative relationship between spectral complexity and spatial support. In higher dimensions, we derive analogous results using pointwise Weyl laws and the eigenfunction restriction estimates on submanifolds.
\end{abstract}

\maketitle

\section{Introduction}

Uncertainty principles are a central theme in harmonic analysis, expressing the fundamental limitation that a function and its frequency representation cannot both be highly localized. In the classical Euclidean setting, this principle is quantified through the Fourier transform, and admits numerous sharp formulations.

On compact Riemannian manifolds, the role of the Fourier transform is played by the spectral decomposition of the Laplace-Beltrami operator. In this setting, uncertainty principles relate the spatial concentration of a function to the concentration of its expansion in terms of Laplace eigenfunctions. This point of view reflects a general paradigm in harmonic analysis: the interaction between geometric and spectral data. 

Let $(M, g)$ be a compact Riemannian manifold without boundary with dimension $n\geq 1$. For the Laplace-Beltrami operator $\Delta_g$ associated with the metric $g$, we have the $L^2(M)$-orthonormal eigenfunctions $e_j$ associated with the eigenvalues $\lambda_j\geq 0$ so that
\begin{align*}
    -\Delta_g e_j=\lambda_j^2 e_j,\quad 0=\lambda_0\leq \lambda_1\leq \cdots.
\end{align*}
In \cite{IosevichMayeliWyman2024Uncertainty}, the authors established an uncertainty principle of the form
\begin{align}
(1-\epsilon-\epsilon')^2 \leq \frac{|E|}{|M|}\cdot \# X_S,
\label{eq:imw-main}
\end{align}
under the additional assumption that the manifold satisfies the pointwise identity
\begin{align}
\sum_{j:\lambda_j=\lambda} |e_j(x)|^2 \equiv \frac{\#\{j:\lambda_j=\lambda\}}{|M|}.
\label{eq:imw-homogeneity}
\end{align}
This condition is known to hold on homogeneous manifolds, and in dimension two it characterizes homogeneity, which is proved in \cite{WangWymanXi2025Surfaces}. For the notations $E, X_S, \epsilon, \epsilon'$, see Theorem \ref{thm:averaged-uncertainty} below.

A natural question arising from this work is whether such uncertainty principles persist beyond the homogeneous setting.

\begin{question}\label{q1}
Can one obtain meaningful uncertainty principles on general compact Riemannian manifolds, without assuming the homogeneity condition above?
\end{question}

The difficulty in addressing this question lies in the failure of uniform distribution of eigenfunctions. On a general manifold, the quantity
\[
\sum_{j:\lambda_j=\lambda} |e_j(x)|^2
\]
may vary significantly with $x$, reflecting geometric inhomogeneities. Thus, the exact identity used in \cite{IosevichMayeliWyman2024Uncertainty} is no longer available.

The main idea of this paper is to replace the homogeneity condition by a quantitative measure of its failure. Given a finite spectral set $S$, we introduce the function
\[
A_S(x)=\sum_{j\in X_S} |e_j(x)|^2,
\]
which measures the local concentration of eigenfunctions. While $A_S(x)$ is constant in the homogeneous case, it encodes deviations from uniformity in general.

Our first result shows that the uncertainty principle of \cite{IosevichMayeliWyman2024Uncertainty} admits a quantitative extension in terms of the maximal deviation of $A_S(x)$ from its average value. More precisely, we prove that
\begin{align}
(1-\epsilon-\epsilon')^2
\leq \frac{|E|}{|M|}\cdot \# X_S \cdot \sup_{x\in E} \frac{A_S(x)}{\frac{\# X_S}{|M|}}.
\label{eq:quant-uncertainty}
\end{align} 
Thus, the classical uncertainty bound is recovered in the homogeneous case, and in general it is controlled by a defect factor measuring spectral inhomogeneity, which we show is optimal in general. In contrast to \cite{IosevichMayeliWyman2024Uncertainty}, which requires an exact identity, this formulation provides a quantitative stability version of the uncertainty principle that applies to arbitrary compact Riemannian manifolds.

We further show that in dimension one, the situation is rigid: every compact connected $1$-dimensional Riemannian manifold without boundary is isometric, after arc-length reparameterization, to a circle of some length, and hence the homogeneity condition holds automatically. This provides a complete positive answer in the one-dimensional case.

In higher dimensions, we investigate the behavior of the defect term using pointwise Weyl laws and related spectral asymptotics, including the case of Schr\"odinger operators with singular potentials. These estimates are closely related to the spectral cluster bounds of Sogge (see, e.g., \cite{Sogge1988Fourier}), which provide a general framework for controlling the concentration of eigenfunctions and hence the size of the defect factor.

This framework provides a unified approach to uncertainty principles on general manifolds, replacing exact symmetry assumptions by quantitative spectral control. In dimension one, this perspective can be made fully quantitative by combining the geometric rigidity of the circle with Fourier-ratio complexity bounds, yielding an explicit relationship between spectral complexity and spatial support. The Fourier ratio, in the sense of scale-localized spectral tails (see the definition below), has been developed in a variety of related contexts as a measure of effective spectral complexity governing uncertainty, approximation, and recovery (see, e.g., \cite{IosevichLiPalssonYavicoli2026, BursteinIosevichNathan2026}). At a conceptual level, these results fit into a broader paradigm in which spectral concentration governs rigidity, stability, and approximation phenomena on compact manifolds. In related work on spectral synthesis (see, e.g., \cite{IosevichMayeliWyman2026Spectral}, \cite{DeodharIosevich2026Spectral}), thin spatial support forces strong restrictions on the distribution of spectral mass, leading to rigidity and quantitative stability results. The uncertainty principles obtained here reflect the same mechanism: exact homogeneity yields a sharp identity, while deviations from homogeneity are captured by a defect factor that measures spectral inhomogeneity and controls the strength of the resulting inequality.

\subsection{Notation}

We write $A\lesssim B$ if $A\leq CB$ for a uniform constant $C$ depending on $M$, $n=\dim M$, or other fixed parameters, but independent of the frequency $\lambda\geq 1$. The notation $A\approx B$ means $A\lesssim B$ and $B\lesssim A$.

\subsection{A quantitative substitute for the homogeneity condition}

In this subsection, we introduce a quantitative substitute for the homogeneity condition \eqref{eq:imw-homogeneity}.

Define
\begin{align}
A_S(x)=\sum_{j\in X_S} |e_j(x)|^2.
\end{align}

Observe that
\begin{align}
\int_M A_S(x)\,dx = \# X_S.
\end{align}

We now derive a general inequality valid on arbitrary compact Riemannian manifolds. The following result provides a quantitative substitute for \eqref{eq:imw-main} on general manifolds.

\begin{theorem}\label{thm:averaged-uncertainty}
Let $(M,g)$ be a compact Riemannian manifold without boundary. Let $E\subset M$ be measurable and $S$ a finite subset of the spectrum. Let $X_S=\{j:\lambda_j\in S\}$.

Suppose that $f\in L^2(M)$ is $L^2$-concentrated on $E$ at level $L$, and $\widehat{f}$ is $L^2$-concentrated on $X_S$ at level $L'$. More precisely, we say that $f$ is $L^2$-concentrated on $E$ at level $L$ if
\[
\int_E |f(x)|^2\,dx \geq (1-\epsilon)\|f\|_{L^2(M)}^2,\quad \text{for some } 0\leq \epsilon<1,
\]
and that $\widehat{f}$ is $L^2$-concentrated on $X_S$ at level $L'$ if
\[
\sum_{j\in X_S} |\widehat{f}(j)|^2 \geq (1-\epsilon')\|f\|_{L^2(M)}^2,\quad \text{for some } 0\leq \epsilon'<1,
\]
where $L=(1-\epsilon^2)^{-1/2}$ and $L'=(1-\epsilon'^2)^{-1/2}$. If $\epsilon, \epsilon'\geq 0$ are small so that $0<1-\epsilon-\epsilon' \leq 1$, then
\begin{align}
(1-\epsilon-\epsilon')^2
\leq \frac{|E|}{|M|}\cdot \# X_S \cdot \sup_{x\in E} \frac{A_S(x)}{\frac{\# X_S}{|M|}}.
\end{align}
\end{theorem}

\begin{proof}
The result follows from \cite[Proposition 1.7]{IosevichMayeliWyman2024Uncertainty} by estimating the integral average of $|e_j(x)|^2$ over $E$ by its supremum on $E$ and using the definition of $A_S(x)$.
\end{proof}

The same proof applies, without change, if one replaces the set $X_S$ by an arbitrary finite index set $I\subset \mathbb{N}$ and defines
\[
A_I(x)=\sum_{j\in I}|e_j(x)|^2.
\]
We shall use this variant below for partial spectral families inside a fixed eigenspace.

\begin{theorem}[Sharpness on $\mathbb{S}^2$]\label{thm:s2-sharpness}
Let $M=\mathbb{S}^2$. For each integer $l\geq 1$, let
\[
f_l(\theta,\varphi)=c_l e^{il\varphi}(\sin \theta)^l,
\]
normalized so that $\|f_l\|_{L^2(\mathbb{S}^2)}=1$. Then:

\begin{enumerate}
\item There exists a set $E_l\subset \mathbb{S}^2$ with
\[
|E_l|\approx l^{-1/2}
\]
such that
\[
\int_{E_l} |f_l(x)|^2\,dx \geq \tfrac{1}{2}.
\]

\item
\[
\sup_{x\in \mathbb{S}^2} |f_l(x)|^2 \approx l^{1/2}.
\]

\item Consequently,
\[
|E_l|\cdot \sup_{x\in E_l} |f_l(x)|^2 \approx 1.
\]
\end{enumerate}

In particular, the defect factor in Theorem \ref{thm:averaged-uncertainty} cannot, in general, be replaced by any smaller-order quantity.
\end{theorem}

\begin{proof}
We write
\[
f_l(\theta,\varphi)=c_l e^{il\varphi}(\sin \theta)^l.
\]

First,
\[
1=\|f_l\|_2^2=2\pi c_l^2 \int_0^\pi (\sin\theta)^{2l+1}\,d\theta \approx c_l^2 l^{-1/2},
\]
so $c_l^2\approx l^{1/2}$, proving the supremum bound.

Next, near the equator $\theta=\pi/2$, we write $\theta=\pi/2+u$. Then
\[
(\sin\theta)^l=(\cos u)^l \approx e^{-c l u^2}.
\]
Thus $|f_l|^2$ is essentially supported where $|u|\lesssim l^{-1/2}$.

Define
\[
E_l=\{x\in \mathbb{S}^2: |\theta-\tfrac{\pi}{2}|\leq C l^{-1/2}\}.
\]
Then $|E_l|\approx l^{-1/2}$ and a Gaussian estimate shows
\[
\int_{E_l} |f_l|^2 \geq \tfrac{1}{2}
\]
for large $C$.

Finally,
\[
|E_l|\cdot \sup_{x\in E_l} |f_l(x)|^2 \approx l^{-1/2}\cdot l^{1/2} \approx 1.
\]
\end{proof}

\begin{remark}
Theorem \ref{thm:s2-sharpness} shows that the defect factor in Theorem \ref{thm:averaged-uncertainty} is not merely a technical artifact. Even on a highly symmetric manifold such as $\mathbb{S}^2$, partial spectral families can exhibit strong spatial concentration, and the quantity
\[
\sup_{x\in E} \frac{A_I(x)}{\frac{\# I}{|M|}}
\]
can grow at exactly the rate needed to balance the shrinking support of the function. Thus, the form of the defect term in Theorem \ref{thm:averaged-uncertainty} is optimal in general.
\end{remark}

\subsubsection{The one-dimensional case}

As we said, we have the following positive answer for Question \ref{q1} on the one-dimensional compact connected Riemannian manifold without boundary.

\begin{theorem}\label{thm:1d-mfld}
Suppose $(M,g)$ is a $1$-dimensional compact connected Riemannian manifold without boundary. Then $(M,g)$ is isometric, after arc-length reparameterization, to a circle of length $L$ for some $L>0$. Consequently, assuming the same hypotheses in Theorem \ref{thm:averaged-uncertainty}, \eqref{eq:imw-homogeneity} holds, and therefore \eqref{eq:imw-main} holds.
\end{theorem}

We shall prove Theorem \ref{thm:1d-mfld} in \S \ref{sec:pf-of-thm:1d-mfld}. We define the Fourier ratio at scale $R\geq 1$ as follows. Let $\{e_n\}$ denote the orthonormal basis of Laplace eigenfunctions on $M$, and write
\[
f=\sum_{n} \widehat{f}(n)e_n,
\]
where $\widehat{f}(n)=\langle f,e_n\rangle$. Then
\[
\FR_R(f)=\frac{\left(\sum_{\lambda_n\geq R} |\widehat{f}(n)|^2\right)^{1/2}}{\|f\|_{L^2(M)}}.
\]

\begin{corollary}[A one-dimensional Fourier-ratio uncertainty bound]
Let $(M,g)$ be a $1$-dimensional compact connected Riemannian manifold without boundary, and let $f\in L^2(M)$ be supported in a measurable set $E\subset M$. Then for every $R\geq 1$,
\[
\FR_R(f)\gtrsim \frac{1}{\sqrt{R\,|E^{R^{-1}}|}},
\]
where $E^{R^{-1}}$ denotes the $R^{-1}$-neighborhood of $E$. Equivalently,
\[
|E^{R^{-1}}|\gtrsim \frac{1}{R\,[\FR_R(f)^2]}.
\]
\end{corollary}

\begin{proof}
By Theorem \ref{thm:1d-mfld}, $(M,g)$ is isometric to a circle, so the problem reduces to the periodic Fourier setting. The stated inequality is then the one-dimensional specialization of the Fourier-ratio lower bound in \cite[Proposition 3.1]{IosevichLiPalssonYavicoli2026}.
\end{proof}

\begin{remark}
The preceding corollary shows that, in dimension one, the uncertainty principle may be interpreted as a lower bound on spectral complexity. Indeed, if $f$ is supported on a thin set $E$, then its high-frequency Fourier-ratio tail cannot be arbitrarily small. In this sense, spatial localization forces spectral complexity: a function supported on a set with small $R^{-1}$-neighborhood must carry a quantitatively nontrivial amount of Fourier mass at frequencies $\geq R$. This places the one-dimensional result in direct continuity with the broader Fourier-ratio program, where the Fourier ratio serves as a complexity parameter governing uncertainty, approximation, and recovery (see, e.g., \cite{IosevichLiPalssonYavicoli2026}).
\end{remark}

\subsubsection{The higher-dimensional case}

For the cases $n=\dim M\geq 2$, we shall use the pointwise Weyl counting function to compute. For this, we briefly review some properties about real-valued potentials. We consider the following time-independent Schr\"odinger operators
\begin{align*}
    H_V=-\Delta_g+V,
\end{align*}
where $V$ is a real-valued potential. Let $\mathcal{K}(M)$ be the set of all Kato class potentials, i.e., $\mathcal{K}(M)$ is the set of all $V$'s such that
\begin{align*}
    \lim_{\delta\to 0+} \sup_{x\in M} \int_{d_g (x, y)<\delta} |V(y)| W_n (d_g (x, y))\:dy=0, \text{ with } W_n (r)=\begin{cases}
        r^{2-n}, & \text{if } n\geq 3, \\
        \log (2+r^{-1}), & \text{if } n=2,
    \end{cases}
\end{align*}
where $d_g$ is the Riemannian distance and $dy$ is the Riemannian volume form on $(M, g)$. It is known that $L^s (M)\subset \mathcal{K}(M)\subset L^1 (M)$ for any $s>\frac{n}{2}$, and both $\mathcal{K}(M)$ and $L^{n/2}(M)$ are critically singular in the sense that they have the same scaling properties and either one is not contained in the other. They are known to be closely related with potentials that appear in physics literature. For details, we refer the reader to \cite{Simon1982Semigroup}.

For $V\in \mathcal{K}(M)\cup L^{n/2}(M)$, by \cite{BlairSireSogge2021Quasimode} and \cite{BlairHuangSireSogge2022UniformSobolev}, $H_V$ is essentially self-adjoint and bounded from below. Adding a uniform positive constant, we may assume that the $H_V$ is essentially self-adjoint and positive. Since $M$ is compact, the spectrum of $H_V$ is discrete as it is for $V\equiv 0$, i.e., $-\Delta_g$. If $V\in \mathcal{K}(M)$, it is also known (by heat kernel estimates) that the associated eigenfunctions are continuous (see, e.g., \cite{LiYau1986Parabolic}, \cite{Sturm1993SchrodingerSemigroups}, \cite{Guneysu2012OnGeneralizedSchrodingerSemigroups}). We can then write
\begin{align*}
    H_V e_j^V=\tau_j^2 e_j^V,
\end{align*}
where the $e_j^V$ are the $L^2$-normalized eigenfunctions of $H_V$ associated with the eigenvalues $\tau_j^2$ with $0\leq \tau_1\leq \tau_2\leq \cdots$ and $\{\tau_j\}$ is discrete. It is known that $\{e_j^V\}$ is also an orthonormal basis on $L^2 (M)$. If $V\equiv 0$, $H_V=-\Delta_g$, and we can write this as
\begin{align*}
    H_0e_j=-\Delta_g e_j=\lambda_j^2 e_j,
\end{align*}
as above. We can then consider the Weyl counting function
\begin{align*}
    N_V (\lambda)=\# \{j: \tau_j\leq \lambda\}=\int_M \sum_{\tau_j\leq \lambda} |e_j^V (x)|^2\:dx,
\end{align*}
and the local Weyl counting function
\begin{align}\label{eq:ptwise-weyl}
    N_x^V (\lambda)=\sum_{\tau_j\leq \lambda} |e_j^V (x)|^2,\quad N_x (\lambda)=\sum_{\lambda_j\leq \lambda} |e_j (x)|^2.
\end{align}
There are substantial results for the (pointwise) Weyl counting function (see, e.g., \cite{Avakumovic1956Eigenfunktionen}, \cite{Levitan1952Asymptotic}, \cite{Levitan1955Expansions}, \cite{Hormander1968SpectralFunction},  \cite{Seeley1978Sharp}, \cite{Seeley1980EstimateNear}, \cite{IosevichWyman2021Weyl}, \cite{HuangSogge2021Weyl}, \cite{HuangZhang2022PtwiseWeyl}, \cite{HuangZhang2023SharpPtwiseWeyl}, \cite{FrankSabin2023SharpWeyl}, \cite{CanzaniGalkowski2023Weyl}, \cite{HuangWangZhang2024Weyl}, and see the references therein). 

Using some of these pointwise Weyl laws, we obtain an potential analogue of \cite[Corollary 1.9]{IosevichMayeliWyman2024Uncertainty}.

\begin{proposition}\label{prop:up-for-hv}
    Let $(M, g)$ be a compact Riemannian manifold without boundary. Let $E$ be a measurable subset of $M$ and let $S$ be a finite subset of the spectrum of $\sqrt{H_V}$. Let $X_S=\{j: \tau_j\in S\}$. Suppose $f$ is $L^2$-concentrated in $E$ at level $L$ with respect to the Riemannian volume measure, and its Fourier transform $\hat{f}$ is $L^2$-concentrated in $X_S$ at level $L'$ with respect to the counting measure, where $L=(1-\epsilon^2)^{-1/2}$ and $L'=(1-\epsilon'^2)^{-1/2}$ as in Theorem \ref{thm:averaged-uncertainty}.
    \begin{enumerate}
        \item Let $V\equiv 0$.
        \begin{enumerate}
            \item Let $\dim M=2$. For any fixed $1-\epsilon_0<a_0<1$ with $0<\epsilon_0\ll 1$ small enough, we have
            \begin{align}\label{eq:2d-up-result}
                a_0(1-\epsilon-\epsilon')^2 \leq \frac{|E|}{|M|} (\# X_S).
            \end{align}
            \item Let $\dim M\geq 3$. Let $S$ be the set
            \begin{align}\label{eq:s-enumeration}
                S=\{s_1, s_2, \cdots, s_k\},\quad \text{where } 1\ll s_1<s_2<\cdots<s_k
            \end{align}
            so that $|S|=k$. Then there is a uniform constant $C>0$ independent of $|E|$ and $\# X_S$ such that
            \begin{align}\label{eq:higher-d-up-result}
                \frac{1}{C}(1-\epsilon-\epsilon')^2\leq \frac{|E|}{|M|}(\# X_S).
            \end{align}
            Here, $C>0$ may depend on $M$ and $S$, since the condition $s_1\gg 1$ is assumed in \eqref{eq:s-enumeration}, but once $s_1\gg 1$ is satisfied, $C>0$ is independent of $\# X_S$.
        \end{enumerate}
        \item For $\dim M=2$, the estimate \eqref{eq:2d-up-result} also holds when $V\in L^n (M)$ (with $a_0$ replaced by $a_V$ possibly).
        \item Let $V\in \mathcal{K}(M)$ and $n=\dim M\geq 2$. Let $S=\{s_1, s_2, \cdots, s_k\}$ with $s_1<s_2<\cdots<s_k$ as above. Then for any fixed $\tilde{\epsilon}>0$, there exists a $\Lambda (\tilde{\epsilon}, V)<\infty$ such that if $s_1>\Lambda (\tilde{\epsilon}, V)$, then there exists a uniform constant $C_V>0$ independent of $|E|$ and $\# X_S$ (but possibly depending on $M$) such that
        \begin{align}\label{eq:up-kato}
            \frac{1}{C_V} (1-\epsilon-\epsilon')^2\leq \frac{|E|}{|M|}(\# X_S).
        \end{align}
    \end{enumerate}
\end{proposition}
We prove Proposition \ref{prop:up-for-hv} in \S \ref{sec:pf-of-up-for-hv}. In Proposition \ref{prop:up-for-hv}, we assumed that the minimum of the set $S$ is sufficiently large, so it may be natural to consider the other opposite cases, for example, the case where the maximum of the set $S$ may be large but bounded above by some constant. This case is sometimes said to be ``band-limited''. Let $\Sigma$ be a $k$-dimensional submanifold of $M$ with $n=\dim M\geq 2$. Let
\begin{align}\label{eq:delta-kp-setup}
    \delta(k, \tilde{p})=\begin{cases}
        \frac{n-1}{4}-\frac{n-2}{2\tilde{p}}, & \text{if } k=n-1 \text{ and } 2\leq \tilde{p}\leq \frac{2n}{n-1}, \\
        \frac{n-1}{2}-\frac{n-1}{\tilde{p}}, & \text{if } k=n-1 \text{ and } \frac{2n}{n-1}\leq \tilde{p}\leq \infty, \\
        \frac{n-1}{2}-\frac{k}{\tilde{p}}, & \text{if } 1\leq k\leq n-2 \text{ and } 2\leq \tilde{p}\leq \infty.
    \end{cases}
\end{align}
We denote by $\mathcal{T}_{\rho} (\Sigma)=\{x\in M: d_g (x, \Sigma)\leq \rho\}$ a tubular neighborhood of $\Sigma$. In our case, it may be natural to focus on $\lambda^{-1}\leq \rho\leq 1$ (see, e.g., \cite[Observation 4]{Tacy2018Constructing}). With this in mind, we first consider a ``unit-length window'' case, i.e., $[\lambda, \lambda+1]$.

\begin{proposition}\label{prop:up-tubular-nbhd}
    Let $(M, g)$ be a smooth compact Riemannian manifold without boundary with $n=\dim M\geq 2$, $P=\sqrt{-\Delta_g}$, and $1\leq p\leq  \infty, 2\leq q\leq \infty$, and $r>1$. Suppose $\Sigma$ is a $k$-dimensional submanifold of $M$, and there is a nontrivial $L^2 (M)$ function $f$ satisfying
    \begin{align}\label{eq:tnbhd-cond}
        \begin{split}
            \|f\|_{L^2 (M)}\lesssim A(\lambda, q, R)\|\mathds{1}_{[\lambda, \lambda+1]} (P)f\|_{L^q (\mathcal{T}_{1/R}(\Sigma))}, \quad \text{for } A(\lambda, q, R)\geq 0 \text{ and } 2\leq q\leq \infty,
        \end{split}
    \end{align}
    where $f=\mathds{1}_{[\lambda, \lambda+1]}(P)f=\sum_{\lambda_j\in [\lambda, \lambda+1]}E_\lambda f=\sum_{\lambda_j\in [\lambda, \lambda+1]} \langle f, e_j \rangle_{L^2 (M)}e_j$, and $R$ and $\lambda$, satisfying $1\leq R\leq \lambda$ and $\lambda\gg 1$, are sufficiently large but fixed. Let $S:=\{\lambda_j\in [\lambda, \lambda+1]: \lambda_j\in \mathrm{spec}(P)\}$. Then there is a uniform constant $C$ independent of $\lambda$ (but possibly depending on $M, \Sigma, p, q$) such that
    \begin{align}\label{eq:up-est-for-tnbhd}
        (\# S)^{\frac{1}{r-1}}|\mathcal{T}_{\frac{1}{R}}(\Sigma)|^{\frac{1}{p'q}}\geq \frac{CR^{\frac{n-k}{q}\left(\frac{1}{p}+\frac{1}{r-1} \right)}}{A(\lambda, q, R)^{\frac{r}{r-1}}\lambda^{\delta(k, pq)+\frac{\delta(k, q)}{r-1}}B(\lambda, k, p, q)},
    \end{align}
    where $\delta(k, \tilde{p})$ is as in \eqref{eq:delta-kp-setup} and
    \begin{align*}
        B(\lambda, k, p, q)=\begin{cases}
            1, & \text{if } (k, pq)\not=(n-2, 2)\not=(k, q), \\
            (\log \lambda)^{\frac{1}{2}}, & \text{if } (k, pq)\not=(n-2, 2)=(k, q), \\
            (\log \lambda)^{\frac{1}{2}} & \text{if } (k, pq)=(n-2, 2)\not=(k, q), \\
            \log \lambda, & \text{if } (k, pq)=(n-2, 2)=(k, q).
        \end{cases}
    \end{align*}
    Here, $|\mathcal{T}_{\frac{1}{R}}(\Sigma)|$ is the Riemannian volume of $\mathcal{T}_{\frac{1}{R}}(\Sigma)$ in $M$, and $p', q'$ are H\"older's conjugates of $p, q$, respectively.
\end{proposition}

We prove Proposition \ref{prop:up-tubular-nbhd} in \S \ref{sec:pf-of-prop-up-tnbhd}. In general, it would be difficult to find a sharpness of (a variant of) \eqref{eq:up-est-for-tnbhd}, but if we assume some additional (strong) assumptions on tori, we may be able to find a sharp example with a $\lambda^\epsilon$-loss (see Remark \ref{remark:sharpness} below). If the eigenvalues are well distributed on $[0, \lambda_0]$ for some $\lambda_0\gg 1$, we have the following summation version of Proposition \ref{prop:up-tubular-nbhd}.

\begin{corollary}\label{cor:up-tubular-nbhd}
    Suppose $S_0$ is a subset of $\{\lambda\in [0, \lambda_0]: \lambda\in \mathrm{spec}(\sqrt{-\Delta_g})\}$ with $\lambda_0\gg 1$, $\lambda$ and $[\lambda, \lambda+1]$ are replaced by $\lambda_0$ and $[0, \lambda_0]$ in \eqref{eq:tnbhd-cond}, and all the other hypotheses in Proposition \ref{prop:up-tubular-nbhd} hold. Then there is a uniform constant $C$ independent of $\lambda_0$ (but possibly depending on $M, \Sigma, p, q$) such that
    \begin{align}\label{eq:large-window-up-est}
        (\# S_0)^{\frac{1}{r-1}}|\mathcal{T}_{\frac{1}{R}}(\Sigma)|^{\frac{1}{p'q}}\geq \frac{CR^{\frac{n-k}{q}\left(\frac{1}{p}+\frac{1}{r-1} \right)}}{A(\lambda_0, q, R)^{\frac{r}{r-1}}\lambda_0^{\delta(k, pq)+\frac{\delta(k, q)+r}{r-1}}B(\lambda_0, k, p, q)}.
    \end{align}
\end{corollary}

We shall prove Corollary \ref{cor:up-tubular-nbhd} in \S \ref{sec:pf-of-cor-up-tnbhd}.

\begin{remark}[Some special cases with or without potentials]
    \ \begin{itemize}
        \item If we choose $(p, q, r)$ so that $r-1=p'q$ (for example, $(p, q, r)=(\infty, 2, 3)$), then we notice that $(\# S)|\mathcal{T}_{1/R} (\Sigma)|$ in Proposition \ref{prop:up-tubular-nbhd} (or $(\# S_0)|\mathcal{T}_{1/R}(\Sigma)|$ in Corollary \ref{cor:up-tubular-nbhd}) cannot be arbitrarily small in the sense of the uncertainty principles.
        \item In \eqref{eq:tnbhd-cond}, if $M$ is the sphere $\mathbb{S}^n$, $f$ is the highest weight spherical harmonics concentrated along a geodesic $\Sigma$, and $R=\lambda^{1/2}$, then we can choose $A(\lambda , 2, R)=1$, though spheres were already considered in \cite[Corollary 1.9]{IosevichMayeliWyman2024Uncertainty}.
        \item In the proof of Proposition \ref{prop:up-tubular-nbhd}, we shall use the fact that the estimates in \cite[Theorem 3]{BurqGerardTzvetkov2007restrictions} (and \cite[Theorem 1.3]{Hu2009lp}) are stable in the sense of \eqref{eq:bgt-hu-est} below. In other words, if we can find such a stability in other circumstances, the estimates \eqref{eq:up-est-for-tnbhd}-\eqref{eq:large-window-up-est} may hold in a broader sense, for example, that $P=\sqrt{-\Delta_g}$ may be replaced by $\sqrt{H_V}$ where $V\in \mathcal{K}(M) \cup L^{n/2}(M)$ (see \cite{BlairPark2025LqEstimates} and \cite{HuangWangZhang2026restriction}).
        \end{itemize}
\end{remark}

\section{Proof of Theorem \ref{thm:1d-mfld}}\label{sec:pf-of-thm:1d-mfld}

\subsection{Model case: a one-dimensional torus with an unusual metric}

    Let $\mathbb{T}=\mathbb{R}/(2\pi \mathbb{Z})(=[-\pi, \pi])$ be a $1$-dimensional torus. We consider the metric $g=(2+\sin x)dx^2$ on $\mathbb{T}$. An arc-length reparameterization
    \begin{align*}
        s(x)=\int_{-\pi}^x \sqrt{2+\sin t}\:dt,\quad L:=s(\pi)=\int_{-\pi}^\pi \sqrt{2+\sin t}\:dt
    \end{align*}
    gives $ds^2=(2+\sin x)dx^2=g$. If we set $\Delta_s=\frac{d^2}{ds^2}$, then $\Delta_s$ is the Laplace-Beltrami operator on $(\mathbb{T}, ds^2)$. Indeed, in our setting, we know that $g=(2+\sin x)dx^2=(s'(x))^2dx^2$, and so, $|g(x)|=\det(g(x))=(s'(x))^2$. With this in mind, by the local formula for the Laplace--Beltrami operator, we have that
    \begin{align*}
        \Delta_g f=\frac{1}{s'(x)}\frac{d}{dx}\left[\left(\frac{1}{s'(x)}\frac{d}{dx} \right)f\right]=\frac{d^2f}{ds^2}=\Delta_s f,
    \end{align*}
    since $\frac{d}{ds}=\frac{dx}{ds}\frac{d}{dx}=\frac{1}{s'(x)}\frac{d}{dx}$. By this, $(\mathbb{T}, g)$ may be regarded as a standard $1$-dimensional torus with length $L$ up to an arc-length parameterization, i.e., $([0, L], ds^2)$, where $0$ and $L$ are identified as in the standard flat torus. With this in mind, we compute eigenfunctions of $\Delta_s$ as we did on the flat torus. We solve an initial value problem
    \begin{align*}
        -\Delta_s f=\lambda^2 f,\quad f(0)=f(L),\quad f'(0)=f'(L),
    \end{align*}
    where we obtain the initial data $f(0)=f(L), f'(0)=f'(L)$ from the periodicity of $([0, L], ds^2)$. By the ODE $-\Delta_s f=\lambda^2 f$, we can write, for constants $\tilde{c}_1, \tilde{c}_2, c_1, c_2$,
    \begin{align*}
        f(s)=\begin{cases}
            \tilde{c}_1+\tilde{c}_2s, & \text{when } \lambda=0, \\
            c_1 e^{i\lambda s}+c_2e^{-i\lambda s}, & \text{when } \lambda>0.
        \end{cases}
    \end{align*}
    When $\lambda=0$, $f$ is constant by the initial data $f(0)=f(L)$ and $f'(0)=f'(L)$, which satisfies the homogeneity condition \eqref{eq:imw-homogeneity} automatically, and so, let us focus on $\lambda>0$. When $\lambda>0$, by the initial data again, if $f$ is nonconstant, we have that $e^{i\lambda L}=e^{-i\lambda L}=1$, that is,
    \begin{align*}
        \lambda=\frac{2\pi n}{L},\quad n\in \mathbb{N}.
    \end{align*}
    The solution $f$ is then
    \begin{align*}
        f(s)=f_n(s)=c_1 e^{i\frac{2\pi ns}{L}}+c_2 e^{-i\frac{2\pi ns}{L}},\quad \lambda=\frac{2\pi |n|}{L},\quad n\in \mathbb{Z}\setminus \{0\}.
    \end{align*}
    Unpacking the definition of $s(x)$, we have
    \begin{align*}
        e^{i\frac{2\pi n s(x)}{L}}=\exp{\left(i\cdot \frac{2\pi n}{L}\int_{-\pi}^x \sqrt{2+\sin t}\:dt\right)}.
    \end{align*}
    We know (by computing the Wronskian) that $\left\{e^{i\frac{2\pi ns}{L}}, e^{-i\frac{2\pi ns}{L}}\right\}$ is linearly independent (on $\mathbb{R}$). Since we are assuming that $\{e_j\}$ is an orthonormal basis for $L^2 (M)$, the $L^2$-normalized eigenfunctions $e_n (x)$ on $(\mathbb{T}, g)$ are of the form
    \begin{align*}
        e_n (x)=\frac{1}{\sqrt{2\pi}}\exp \left(i\frac{2\pi n}{L} \int_{-\pi}^x \sqrt{2+\sin t}\:dt \right),\quad \text{where} \quad L=\int_{-\pi}^\pi \sqrt{2+\sin t}\:dt.
    \end{align*}
    Since we have the eigenvalues $\lambda_n=\frac{2\pi |n|}{L}$ for $n\in \mathbb{Z}\setminus \{0\}$, their multiplicities are $\# \{j: \lambda_j=\lambda_n\}=2$, and so,
    \begin{align*}
\sum_{j:\lambda_j=\lambda_n} |e_j (x)|^2=2\cdot \frac{1}{(\sqrt{2\pi})^2}=\frac{2}{2\pi}=\frac{\# \{j:\lambda_j=\lambda_n\}}L.
    \end{align*}
    This also satisfies \eqref{eq:imw-homogeneity}, and hence, \eqref{eq:imw-main}.

\begin{remark}
    By the above computation, we can view $\mathbb{T}=([0, L], ds^2)$ as a $1$-dimensional flat torus with length $L$, and so, in this case, $(\mathbb{T}, ds^2)$ is homogeneous, but $(\mathbb{T}, g)$ itself may not be homogeneous. Indeed, suppose $(\mathbb{T}, g)$ is homogeneous. Then for any $p, q\in \mathbb{T}$, there is an isometry $F$ such that $F(p)=q$. We now recall a standard exercise in Riemannian geometry (cf. \cite[Problem 2-1]{Lee2018secondEd}):
    \begin{align}\label{eq:1d-flat}
        \text{Every Riemannian $1$-manifold is flat.}
    \end{align}
    Let $g_e$ be the usual Euclidean metric on $\mathbb{R}$. By \eqref{eq:1d-flat}, for any $p, q\in (\mathbb{T}, g)$, there are local isometries $\varphi$ and $\psi$ such that $\varphi$ and $\psi$ map an open set containing $p$ and $q$ to an open set containing $\varphi(p)=0$ and $\psi(q)=0$ in $(\mathbb{R}, g_e)$, respectively. If we consider a transition $F\circ \varphi^{-1}=\psi^{-1} \circ \tilde{F}$, then by a standard computation in local coordinates (cf. \cite[(3.10)]{Lee2013SM}), we have
    \begin{align*}
        dF_p \left(\left.\frac{\partial}{\partial x}\right|_p \right)=\frac{\partial \tilde{F}}{\partial x} (\varphi(p))\left. \frac{\partial}{\partial x}\right|_{F(p)}.
    \end{align*}
    Locally, the transition map $\tilde{F}=\psi\circ F\circ \varphi^{-1}$ is well-defined. Since $\psi$ and $\varphi$ are local isometries and $F$ is a global isometry, $\tilde{F}$ is a local isometry from $(\mathbb{R}, g_e)$ to $(\mathbb{R}, g_e)$ itself. Thus, the Jacobian $\frac{\partial \tilde{F}}{\partial x}$ is an orthogonal matrix on $\mathbb{R}$, i.e., $\frac{\partial \tilde{F}}{\partial x}=\pm 1$. Since $F$ is an isometry, we know $F^* g=g$. With this in mind, we have
    \begin{align*}
        (F^* g)_p \left(\left.\frac{\partial}{\partial x}\right|_p, \left. \frac{\partial}{\partial x}\right|_p \right)=g\left(\left.\frac{\partial}{\partial x}\right|_p, \left.\frac{\partial}{\partial x}\right|_p \right)=(2+\sin p)dx^2 \left(\left.\frac{\partial}{\partial x}\right|_p, \left.\frac{\partial}{\partial x}\right|_p\right)=2+\sin p.
    \end{align*}
    On the other hand, by $\frac{\partial \tilde{F}}{\partial x}=\pm 1$, we have
    \begin{align*}
        (F^* g)_p \left(\left.\frac{\partial}{\partial x}\right|_p, \left.\frac{\partial}{\partial x}\right|_p \right)&=g_{F(p)} \left(dF_p \left(\left.\frac{\partial}{\partial x}\right|_p \right), dF_p \left(\left.\frac{\partial}{\partial x}\right|_p \right) \right) \\
        &=g_{F(p)} \left(\frac{\partial \tilde{F}}{\partial x}(\varphi (p))\left.\frac{\partial}{\partial x}\right|_{F(p)}, \frac{\partial \tilde{F}}{\partial x}(\varphi (p))\left.\frac{\partial}{\partial x}\right|_{F(p)} \right) \\
        &=\left|\frac{\partial \tilde{F}}{\partial x}(\varphi (p)) \right|^2 g_{F(p)} \left(\left.\frac{\partial}{\partial x}\right|_{F(p)}, \left.\frac{\partial}{\partial x}\right|_{F(p)} \right)=2+\sin (F(p)).
    \end{align*}
    Combining these two, we have $2+\sin p=2+\sin (F(p))$ for any $p, q\in (\mathbb{T}, g)$ where $q=F(p)$. However, if we choose $F$ so that $p=0$ and $q=F(p)=\frac{\pi}{2}$, we have $2+\sin p\not=2+\sin (F(p))$, which is a contradiction. Hence, $(\mathbb{T}, g)$ itself is not homogeneous, but the metric can be rescaled/reparameterized as $ds^2$ so that $(\mathbb{T}, ds^2)$ is homogeneous, which in turn implies \eqref{eq:imw-homogeneity}, and hence, \eqref{eq:imw-main}.
\end{remark}

\subsection{General one-dimensional cases}
We now show Theorem \ref{thm:1d-mfld} by using the argument of the model case above.  A compact connected $1$-dimensional manifold without boundary is diffeomorphic to $\mathbb{S}^1$. Let $\gamma_1:\mathbb{S}^1\to M$ be a smooth parametrization. By a standard exercise in manifold theory, there is a diffeomorphism $\gamma_2:\mathbb{R}/(2\pi \mathbb{Z})\to \mathbb{S}^1$ between the flat torus $\mathbb{R}/(2\pi \mathbb{Z})$ and $\mathbb{S}^1$. Let $\gamma=\gamma_1\circ \gamma_2: \mathbb{R}/(2\pi \mathbb{Z})\to M$. Pulling back the metric $g$ by $\gamma$, we obtain a Riemannian metric on $\mathbb{S}^1$ of the form
\begin{align*}
    \gamma^*g=h(\theta)\,d\theta^2,
\end{align*}
where $h$ is a smooth positive function. Since $\gamma^*g$ is a metric on the flat torus, $h$ is $2\pi$-periodic, so we can use the proof the model case in what follows.

Define the arclength parameter
\[
s(\theta)=\int_0^\theta \sqrt{h(t)}\,dt.
\]
Since $h>0$, the map $s$ is strictly increasing, and its total length is
\[
L=\int_0^{2\pi}\sqrt{h(t)}\,dt.
\]
Because $h$ is $2\pi$-periodic, the parameter $s$ identifies $\mathbb{S}^1$ with the flat circle $\mathbb{R}/L\mathbb{Z}$, and in this parameter the metric becomes
\[
ds^2.
\]
Thus $(M,g)$ is isometric to a circle of length $L$.

The eigenfunctions of the Laplace-Beltrami operator on $\mathbb{R}/L\mathbb{Z}$ are
\[
e_n(s)=L^{-1/2}e^{2\pi i n s/L}, \quad n\in \mathbb{Z},
\]
with eigenvalues
\[
\lambda_n=\frac{2\pi |n|}{L}.
\]
Hence
\[
\sum_{j:\lambda_j=\lambda}|e_j(s)|^2=\frac{\#\{j:\lambda_j=\lambda\}}{L}
=\frac{\#\{j:\lambda_j=\lambda\}}{|M|},
\]
so \eqref{eq:imw-homogeneity} holds. The conclusion \eqref{eq:imw-main} now follows from \cite[Corollary 1.9]{IosevichMayeliWyman2024Uncertainty}.

\begin{remark}
    For the $2$-dimensional torus with unusual metrics, we cannot apply \eqref{eq:1d-flat} directly, so we cannot say exactly the same for $2$-dimensional tori. However, if we consider $(\mathbb{T}^2, g_2)$ where $g_2=dx^2+(2+\sin y)dy^2$, then we may use the separation of variable technique from ODEs, and may argue similarly as in the $1$-dimensional model case, and thus, one can see that $(\mathbb{T}^2, g_2)$ itself may not be homogeneous even though $(\mathbb{T}^2, dx^2+ds^2)$ is homogeneous, which satisfies \eqref{eq:imw-homogeneity}. This does not violate \cite[Theorem 1.2]{WangWymanXi2025Surfaces}, since rescaling the metric is already considered in the proof of \cite[Theorem 1.2]{WangWymanXi2025Surfaces} (see \cite[Section 2]{WangWymanXi2025Surfaces}), so one can say that \cite[Theorem 1.2]{WangWymanXi2025Surfaces} holds up to a rescaling/reparameterization of metric.
    
    However, if we want to consider general 2- or higher-dimensional analogues of Theorem \ref{thm:1d-mfld}, there may be some metrics highly intertwined with one another, so we may not be able to use the separation of variables easily, which makes our computations difficult. Thus, in higher-dimensional cases, we shall use other methods, e.g., in this paper, we shall use the pointwise Weyl laws and eigenfunction restriction estimates.
\end{remark}

\section{Proof of Proposition \ref{prop:up-for-hv}}\label{sec:pf-of-up-for-hv}

In this section, for simplicity, we may assume $\lambda\in \mathrm{spec}(\sqrt{-\Delta_g})$ (or $\lambda\in \mathrm{spec}(\sqrt{H_V})$ when $V$ is a singular potential). The other case can be treated similarly. We first recall some versions of the pointwise Weyl laws. For $V\equiv 0$, by \cite{Avakumovic1956Eigenfunktionen}, we know
\begin{align}\label{eq:avakumovic-ptwise-weyl}
    N_x (\lambda)=\sum_{\lambda_j\leq \lambda} |e_j (x)|^2=(2\pi)^{-n}\omega_n \lambda^n+O(\lambda^{n-1}),
\end{align}
where $\omega_n=\frac{\pi^{\frac{n}{2}}}{\Gamma\left(\frac{n}{2}+1\right)}$ is the volume of the unit ball in $\mathbb{R}^n$. For singular potentials $V$, we recall the following theorem \cite[Theorem 1]{HuangZhang2022PtwiseWeyl}.

\begin{theorem}[\cite{HuangZhang2022PtwiseWeyl}]\label{thm:hz-ptwise-weyl}
    Suppose $n\geq 2$, $N_x^V (\lambda)$ is as in \eqref{eq:ptwise-weyl}, $\omega_n$ is the volume of the unit ball in $\mathbb{R}^n$, and $C_V>0$ is a constant independent of $\lambda$ and $\epsilon$.
    \begin{enumerate}
        \item If $V\in \mathcal{K}(M)$, then for any fixed $\epsilon>0$ there exists a $\Lambda(\epsilon, V)<\infty$ such that for $\lambda>\Lambda(\epsilon, V)$, we have
        \begin{align}\label{eq:hz-kato-ptwise}
            \sup_{x\in M} |N_x^V (\lambda)-(2\pi)^{-n} \omega_n \lambda^n|\leq C_V \epsilon \lambda^n,
        \end{align}
        \item If $V\in L^n (M)$, then for any fixed $\epsilon>0$ there exists a $\Lambda(\epsilon, V)<\infty$ such that for $\lambda>\Lambda(\epsilon, V)$, we have
        \begin{align}\label{eq:hz-ln-ptwise}
            \sup_M |N_x^V (\lambda)-(2\pi)^{-n} \omega_n \lambda^n|\leq C_V \lambda^{n-1}.
        \end{align}
    \end{enumerate}
\end{theorem}

\subsection{The two dimensional case}
We first show \eqref{eq:2d-up-result}. \cite[Corollary 1.9]{IosevichMayeliWyman2024Uncertainty} already obtained a better result by assuming \eqref{eq:imw-homogeneity}, and so, we now suppose that \eqref{eq:imw-homogeneity} does not hold, i.e., $\sum_{j:\lambda_j=\lambda} |e_j (x)|^2$ may depend on $x\in M$. Then the pointwise Weyl counting function $N_x (\lambda)$ may also depend on $x\in M$. In this case, by \eqref{eq:avakumovic-ptwise-weyl}, since we know that $(2\pi)^{-n}\omega_2 \lambda^2$ is independent of $x\in M$, the remainder $O(\lambda)$ may depend on $x\in M$. With this in mind, we write
\begin{align*}
    N_x (\lambda)=\sum_{\lambda_j\leq \lambda} |e_j (x)|^2=(2\pi)^{-2} \omega_2 \lambda^2+\lambda c(x),
\end{align*}
where $c(x)=c(x, \lambda)=O(1)$. Since the set of eigenvalues $\{\lambda_j\}_{j=1}^\infty$ is discrete, for each $\lambda>1$, there is an $0<\epsilon_1=\epsilon_1 (\lambda)\ll 1$ such that
\begin{align}\label{eq:diff-of-nx-lambda}
    \begin{split}
        \sum_{\lambda_j=\lambda} |e_j (x)|^2&=N_x (\lambda)-N_x (\lambda-\epsilon_1)=(2\pi)^{-2} \omega_2 (\lambda^2-(\lambda-\epsilon_1)^2)+(\lambda-(\lambda-\epsilon_1))c(x) \\
        &=2(2\pi)^{-2} \omega_2 \epsilon_1 \lambda+\epsilon_1 [c(x)-(2\pi)^{-2}\epsilon_1 \omega_2]
    \end{split}
\end{align}
Integrating both sides over $M$ yields
\begin{align}\label{eq:lambda-multi-for-n2}
    \#\{j: \lambda_j=\lambda\}=\int_M \sum_{\lambda_j=\lambda} |e_j (x)|^2\:dx=2(2\pi)^{-2}\omega_2 \epsilon_1 \lambda |M|+\epsilon_1 \int_M [c(x)-(2\pi)^{-2}\epsilon_1 \omega_2]\:dx,
\end{align}
from which it follows that
\begin{align*}
    2(2\pi)^{-2} \omega_2 \epsilon_1 \lambda=\frac{\#\{j: \lambda_j=\lambda\}}{|M|}-\frac{\epsilon_1}{|M|}\int_M (c(x)-(2\pi)^{-2}\epsilon_1 \omega_2)\:dx.
\end{align*}
Putting this into \eqref{eq:diff-of-nx-lambda}, we can write
\begin{align}\label{eq:ptwise-lambda-with-error}
    \sum_{\lambda_j=\lambda} |e_j (x)|^2=\frac{\#\{j: \lambda_j=\lambda\}}{|M|}+\epsilon_1 c_1 (x, \epsilon_1),
\end{align}
where
\begin{align*}
    c_1 (x, \epsilon_1)=c_1 (x, \epsilon_1, \lambda)=c(x)-(2\pi)^{-2} \epsilon_1 \omega_2-\frac{1}{|M|} \int_M (c(x)-(2\pi)^{-2} \epsilon_1 \omega_2)\:dx.
\end{align*}
We note that $|c_1 (x, \epsilon_1)|\lesssim 1$. Since $S$ is finite, we can write $S=\{s_1, s_2, \cdots, s_k\}$. With this in mind, by \cite[Proposition 1.7]{IosevichMayeliWyman2024Uncertainty} and \eqref{eq:ptwise-lambda-with-error}, we have
\begin{align}\label{eq:upper-bd-of-1ee2}
    \begin{split}
        (1-\epsilon-\epsilon')^2&\leq \sum_{j\in X_S} \int_E |e_j (x)|^2\:dx=\int_E\sum_{j\in X_S} |e_j (x)|^2\:dx=\int_E \sum_{l=1}^k \sum_{\lambda_j=s_l} |e_j (x)|^2\:dx \\
        &=\int_E \sum_{l=1}^k \left[\frac{\#\{j:\lambda_j=s_l\}}{|M|}+\epsilon_1 c_1 (x, \epsilon_1)\right]dx=\int_E \left[\frac{\# X_S}{|M|}+\epsilon_1 c_1 (x, \epsilon_1) |S|\right]dx \\
        &=\frac{|E|}{|M|}(\# X_S)+\epsilon_1 \left(\int_E c_1 (x, \epsilon_1)\:dx \right)|S|.
    \end{split}
\end{align}
Since $|c_1 (x, \epsilon_1)|\lesssim 1$, we may assume $|c_1 (x, \epsilon_1)|\leq C$ for some uniform constant $C>0$. We also note that $|S|\leq \# X_S$, where the equality holds if and only if each $\lambda_j\in S$ has multiplicity $1$. Since $|S|=k$ is finite (even if it is large), given any $0<a_0<1$, one can choose $0<\epsilon_1\ll 1$ small enough so that
\begin{align*}
    \epsilon_1 \left|\int_E c_1 (x, \epsilon_1)\:dx\right| |S|\leq C\epsilon_1 |E|\cdot \# X_S\leq \frac{1-a_0}{a_0} \frac{|E|}{|M|}(\# X_S).
\end{align*}
Combining this with \eqref{eq:upper-bd-of-1ee2}, we obtain \eqref{eq:2d-up-result}. When $V\in L^n (M)$, one can also obtain \eqref{eq:2d-up-result}, if we use \eqref{eq:hz-ln-ptwise} with the argument above, instead of using \eqref{eq:avakumovic-ptwise-weyl}.

\subsection{The higher-dimensional case with vanishing potentials}
We show \eqref{eq:higher-d-up-result}. Let $n=\dim M\geq 3$. Suppose $\lambda\gg 1$. As in the $2$-dimensional case, by \eqref{eq:avakumovic-ptwise-weyl} one can write
\begin{align*}
   N_x (\lambda)=\sum_{\lambda_j\leq \lambda} |e_j (x)|^2=(2\pi)^{-n} \omega_n \lambda^n+\lambda^{n-1}c(x),
\end{align*}
As before, for each $\lambda$, there is an $0<\epsilon_1=\epsilon_1 (\lambda)\ll 1$ such that
\begin{align}\label{eq:diff-for-v0-higher}
    \begin{split}
        \sum_{\lambda_j=\lambda} |e_j (x)|^2&=N_x (\lambda)-N_x (\lambda-\epsilon_1)=(2\pi)^{-n}\omega_n (\lambda^n-(\lambda-\epsilon_1)^n)+c(x)(\lambda^{n-1}-(\lambda-\epsilon_1)^{n-1}) \\
        &=:(2\pi)^{-n}\omega_n (n\epsilon_1) \lambda^{n-1}+\sum_{k=0}^{n-2} c_{n, k} (x) \lambda^k \epsilon_1^{n-k-1},
    \end{split}
\end{align}
where
\begin{align*}
    c_{n, k}(x)=\left[(2\pi)^{-n} \omega_n \begin{pmatrix}
        n \\
        k
    \end{pmatrix}\epsilon_1-c(x)\begin{pmatrix}
        n-1 \\
        k
    \end{pmatrix}\right] (-1)^{n-1-k}.
\end{align*}
Since $|c_{n, k}|\lesssim 1$, there is a uniform constant $C>0$ such that
\begin{align}\label{eq:cnk-error-for-v0}
    \left|\sum_{k=0}^{n-2} c_{n, k}(x)\lambda^k \epsilon_1^{n-k-1} \right|\leq C\lambda^{n-2}\epsilon_1,\quad \text{whenever } \lambda\gg 1.
\end{align}
On the other hand, integrating the sides in \eqref{eq:diff-for-v0-higher}, one can obtain the following identity analogous to \eqref{eq:lambda-multi-for-n2}.
\begin{align*}
    \frac{\#\{j: \lambda_j=\lambda\}}{|M|}=(2\pi)^{-n} \omega_n (n\epsilon_1) \lambda^{n-1}+\frac{1}{|M|} \int_M \left[\sum_{k=0}^{n-2} c_{n, k} (x)\lambda^k \epsilon_1^{n-k-1}\right]dx.
\end{align*}
For $\lambda\gg 1$ and $0<\epsilon_1\ll 1$, by this and \eqref{eq:cnk-error-for-v0}, we have
\begin{align*}
    \frac{\#\{j: \lambda_j=\lambda\}}{|M|}\approx (2\pi)^{-n} \omega_n (n\epsilon_1) \lambda^{n-1}.
\end{align*}
Putting this into \eqref{eq:diff-for-v0-higher}, by \eqref{eq:cnk-error-for-v0}, for $0<\epsilon_0\ll 1$, we have that
\begin{align}\label{eq:lambda-multi-approx-for-n-geq-3}
    \sum_{\lambda_j=\lambda} |e_j (x)|^2 \approx (2\pi)^{-n} \omega_n (n\epsilon_1) \lambda^{n-1}\approx \frac{\#\{j: \lambda_j=\lambda\}}{|M|},\quad \text{for } \lambda\gg 1.
\end{align}
By \eqref{eq:s-enumeration}, we are assuming $s_1\gg 1$, and so, by \eqref{eq:lambda-multi-approx-for-n-geq-3} and \cite[Proposition 1.7]{IosevichMayeliWyman2024Uncertainty}, we have that
\begin{align}\label{eq:up-higherd}
    \begin{split}
        (1-\epsilon-\epsilon')^2&\leq \sum_{j\in X_S} \int_E |e_j (x)|^2\:dx=\int_E \sum_{j\in X_S} |e_j (x)|^2\:dx \\
        &=\int_E \sum_{l=1}^k \sum_{\lambda_j=s_l} |e_j (x)|^2\:dx \approx \int_E \sum_{l=1}^k \frac{\# \{j: \lambda_j=s_l\}}{|M|}\:dx=\int_E \frac{\# X_S}{|M|}\:dx=\frac{|E|}{|M|}(\# X_S).
    \end{split}
\end{align}
Hence,
\begin{align*}
    (1-\epsilon-\epsilon')^2 \leq C\frac{|E|}{|M|}(\# X_S),
\end{align*}
where $C>0$ is a uniform constant independent of $E$. The constant $C>0$ is depending on $S$, since we need the condition $\min S\gg 1$, but once the condition $\min S\gg 1$ is satisfied, $C$ is independent of $\# X_S$. This proves \eqref{eq:higher-d-up-result}.

\subsection{The Kato class potential case}
We show \eqref{eq:up-kato}. If we look into the proof of \cite[Proposition 1.7]{IosevichMayeliWyman2024Uncertainty} in \cite[\S2.2]{IosevichMayeliWyman2024Uncertainty}, the proof uses the fact that the set of all eigenfunctions $\{e_j\}$ of $-\Delta_g$ is an orthonormal basis on $L^2 (M)$. Since the eigenfunctions $\{e_j^V\}$ of $H_V=-\Delta_g+V$ where $V\in \mathcal{K}(M)$ also form an orthonormal basis on $L^2 (M)$, replacing $\{e_j\}$ by $\{e_j^V\}$ in \cite[\S2.2]{IosevichMayeliWyman2024Uncertainty}, one can obtain the following analogue of \cite[Proposition 1.7]{IosevichMayeliWyman2024Uncertainty} (see also \cite[Remark 1.8]{IosevichMayeliWyman2024Uncertainty}).

\begin{proposition}\label{prop:kato-analogue}
    Let $S$ be a finite subset of the set of eigenvalues of $\sqrt{H_V}$. Let $X_S=\{j: \tau_j\in S\}$. Suppose that $f\in L^2 (M)$ is not identically $0$ and $f$ is $L^2$-concentrated in $E\subset M$ at level $L$ with respect to the Riemannian volume density. Suppose also that $\hat{f}$ is $L^2$-concentrated on $X_S$ at level $L'$ with respect to the counting measure. Then
    \begin{align*}
        \left(\frac{1}{\# X_S} \sum_{j\in X_S} \frac{1}{|E|} \int_E |e_j^V (x)|^2\:dx \right)^{-1}\leq (1-\epsilon-\epsilon')^{-2} |E|(\# X_S),
    \end{align*}
    where, $L=(1-\epsilon^2)^{-\frac{1}{2}}$ and $L'=(1-\epsilon'^2)^{-\frac{1}{2}}$, for some $0\leq \epsilon, \epsilon'<1$.
\end{proposition}

We repeat the argument of the case $V\equiv 0$. Let $\tilde{\epsilon}>0$ be arbitrary. For each $\lambda>1$, there is an $0<\epsilon_1=\epsilon_1 (\lambda)\ll 1$ such that
\begin{align*}
    \sum_{\lambda_j=\lambda} |e_j^V (x)|^2=N_x^V (\lambda)-N_x^V (\lambda-\epsilon_1).
\end{align*}
By \eqref{eq:hz-kato-ptwise}, there exists a $\Lambda(\tilde{\epsilon}, V)<\infty$ such that for $\lambda>\Lambda(\tilde{\epsilon}, V)$, we can write
\begin{align*}
    N_x^V (\lambda)=(2\pi)^{-n}\omega_n \lambda^n+\lambda^n \tilde{\epsilon} c_V (x),\quad x\in M,
\end{align*}
where $|c_V (x)|\leq C_V'$.

Combining these two, we can write
\begin{align}\label{eq:ptw-ejv2-comp}
    \begin{split}
        \sum_{\lambda_j=\lambda}|e_j^V (x)|^2&=(2\pi)^{-n} \omega_n (\lambda^n-(\lambda-\epsilon_1)^n)+\tilde{\epsilon}c_V(x) (\lambda^n-(\lambda-\epsilon_1)^n) \\
        &=((2\pi)^{-n}\omega_n+\tilde{\epsilon}c_V (x))[n\lambda^{n-1}\epsilon_1+c_{n, \epsilon_1}],
    \end{split}
\end{align}
where
\begin{align*}
    c_{n, \epsilon_1}=c_{n, \epsilon_1}(\lambda)=\sum_{k=0}^{n-2} \begin{pmatrix}
        n \\
        k
    \end{pmatrix} \lambda^k (-1)^{n-1-k} \epsilon_1^{n-k}.
\end{align*}
We note that $|c_{n, \epsilon_1}|\leq C\lambda^{n-2}\epsilon_1^2$, where $C>0$ is independent of $\epsilon_1$ and $\lambda$. Choosing $\tilde{\epsilon}>0$ sufficiently small and $0<\epsilon_1\ll 1$ is sufficiently small but fixed, by \eqref{eq:ptw-ejv2-comp}, we have
\begin{align*}
    \sum_{\tau_j=\lambda} |e_j^V (x)|^2\approx (2\pi)^{-n} \omega_n (n\epsilon_1)\lambda^{n-1},\quad \text{for } \lambda\gg 1.
\end{align*}
Integrating this over $M$, we obtain
\begin{align*}
    \#\{j:\tau_j=\lambda\}\approx (2\pi)^{-n}\omega_n (n\epsilon_1)\lambda^{n-1}|M|.
\end{align*}
Combining these two yields
\begin{align*}
    \sum_{\tau_j=\lambda} |e_j^V (x)|^2\approx \frac{\#\{j: \tau_j=\lambda\}}{|M|},\quad \text{for } \lambda\gg 1.
\end{align*}
By this, Proposition \ref{prop:kato-analogue}, and the computation as in \eqref{eq:up-higherd}, we obtain
\begin{align*}
    (1-\epsilon-\epsilon')^2\leq \sum_{j\in X_S} \int_E |e_j^V (x)|^2\:dx \approx \frac{|E|}{|M|}(\# X_S).
\end{align*}
This completes the proof of \eqref{eq:up-kato}.

\section{Proof of Proposition \ref{prop:up-tubular-nbhd}}\label{sec:pf-of-prop-up-tnbhd}
For simplicity, we focus on the case where $(k, pq)\not=(n-2, 2)\not=(k, q)$. The other cases follow similarly. Since the restriction estimates are local, we may assume in local coordinates so that $\Sigma$ may be locally identified as
\begin{align*}
    \{(x, 0)\in \mathbb{R}^k\times \mathbb{R}^{n-k}: |x|\leq 1\},
\end{align*}
and $\mathcal{T}_{1/R}(\Sigma)$ can be covered by
\begin{align*}
    \bigcup_{|y|\leq \frac{C}{R}} \Sigma_y,\quad \text{where } \Sigma_y:=\{(x, y)\in \mathbb{R}^k \times \mathbb{R}^{n-k}: |x|\leq 2\}.
\end{align*}
By \cite{BurqGerardTzvetkov2007restrictions} and \cite{Hu2009lp}, we know that
\begin{align*}
    \|\mathds{1}_{[\lambda, \lambda+1]}(P) f\|_{L^p (\Sigma)}\leq \begin{cases}
        C\lambda^{\delta(k, p)}\|f\|_{L^2 (M)},& \text{for } (k, p)\not=(n-2, 2), \\
        C\lambda^{\delta(k, p)}(\log \lambda)^{1/2}\|f\|_{L^2 (M)},& \text{for } (k, p)=(n-2, 2).
    \end{cases}
\end{align*}
We note that if $f=e_\lambda$ in this above estimate, the estimate gives us eigenfunction restriction estimates on $\Sigma$. By the proof of \cite{BurqGerardTzvetkov2007restrictions} (see also \cite{Hu2009lp}), the constant $C$ in the above estimate is uniform under small smooth perturbations of $\Sigma$ (the constant may be large, but independent of $\lambda$, if needed), and so, we have that
\begin{align}\label{eq:bgt-hu-est}
    \|\mathds{1}_{[\lambda, \lambda+1]} (P)\|_{L^2 (M)\to L^q (\Sigma_y)}\lesssim \begin{cases}
        \lambda^{\delta(k, q)},& \text{for } (k, q)\not=(n-2, 2) \text{ and } 2\leq q\leq \infty, \\
        \lambda^{\delta(k, q)}(\log \lambda)^{1/2},& \text{for } (k, q)=(n-2, 2).
    \end{cases} 
\end{align}
For simplicity, we focus only on the no log-loss cases, which corresponds to $(k, pq)\not=(n-2, 2)\not=(k, q)$, considering the argument below. By H\"older's inequality, we have
\begin{align}\label{eq:1s0p-tnbhd-est}
    \begin{split}
        \|\mathds{1}_{[\lambda, \lambda+1]} (P)f\|_{L^q (\mathcal{T}_{1/R}(\Sigma))}^q&\lesssim \int_{|y|\leq C/R} \|\mathds{1}_{[\lambda, \lambda+1]} (P)f\|_{L^q (\Sigma_y)}^q\:dy \\
        &\leq \int_{|y|\leq C/R} \int_{\Sigma_y} 1\cdot |\mathds{1}_{[\lambda, \lambda+1]} (P)f(x)|^q\:d\Sigma_y (x)\:dy \\
        &\leq \int_{|y|\leq C/R} \left(\int_{\Sigma_y} 1^{p'}\:d\Sigma_y (x) \right)^{\frac{1}{p'}}\left(\int_{\Sigma_y} |\mathds{1}_{[\lambda, \lambda+1]} (P) f(x)|^{pq}\:d\Sigma_y (x) \right)^{\frac{1}{p}}\:dy \\
        &=\int_{|y|\leq C/R} |\Sigma_y|^{\frac{1}{p'}} \|\mathds{1}_{[\lambda, \lambda+1]}(P) f\|_{L^{pq}(\Sigma_y)}^q\:dy \\
        &\leq \left(\sup_{|y|\leq C/R} |\Sigma_y|^{\frac{1}{p'}} \right)\int_{|y|\leq C/R} \|\mathds{1}_{[\lambda, \lambda+1]} (P) f\|_{L^{pq}(\Sigma_y)}^q\:dy,
    \end{split}
\end{align}
where $|\Sigma_y|$ is the (local) area of $\Sigma_y$ in local coordinates $\mathbb{R}^k \times \{y\}\simeq \mathbb{R}^k$. By \eqref{eq:bgt-hu-est} and \eqref{eq:1s0p-tnbhd-est}, we have that
\begin{align}\label{eq:1s0p-tnbhd-short-comp}
    \begin{split}
        \|\mathds{1}_{[\lambda, \lambda+1]}(P)f\|_{L^q (\mathcal{T}_{1/R} (\Sigma))}^q&\lesssim \left(\sup_{|y|\leq C/R} |\Sigma_y|^{\frac{1}{p'}} \right)\int_{|y|\leq C/R} \lambda^{q\delta(k, pq)}\|f\|_{L^2 (M)}^q\:dy \\
        &\lesssim \left(\sup_{|y|\leq C/R} |\Sigma_y|^{\frac{1}{p'}} \right) \left(\frac{1}{R} \right)^{n-k} \lambda^{q\delta(k, pq)}\|f\|_{L^2 (M)}^q,
    \end{split}
\end{align}
and thus,
\begin{align}\label{eq:lower-bd-for-f-tnbhd}
    \frac{R^{\frac{n-k}{q}}}{\left(\sup_{|y|\leq C/R} |\Sigma_y|^{\frac{1}{p'q}} \right)\lambda^{\delta(k, pq)}}\lesssim \frac{\|f\|_{L^2 (M)}}{\|\mathds{1}_{[\lambda, \lambda+1]} (P)f\|_{L^q (\mathcal{T}_{1/R}(\Sigma))}}.
\end{align}
We now use the condition \eqref{eq:tnbhd-cond}. By \eqref{eq:bgt-hu-est},
\begin{align}\label{eq:1st-comp-for-upbd-f-tnbhd}
    \begin{split}
        \|\mathds{1}_{[\lambda, \lambda+1]}(P)f\|_{L^q (\mathcal{T}_{1/R}(\Sigma))}&\lesssim \left(\int_{|y|\leq C/R} \|\mathds{1}_{[\lambda, \lambda+1]}(P) f\|_{L^q (\Sigma_y)}^q\:dy \right)^{\frac{1}{q}} \\
        &\lesssim \left(\frac{1}{R} \right)^{\frac{n-k}{q}} \lambda^{\delta(k, q)} \|\mathds{1}_{[\lambda, \lambda+1]}(P)f\|_{L^2 (M)} \leq \left(\frac{1}{R} \right)^{\frac{n-k}{q}} \lambda^{\delta(k, q)} (\# S)\|f\|_{L^2 (M)}.
    \end{split}
\end{align}
By this and \eqref{eq:tnbhd-cond}, we have
\begin{align*}
    \begin{split}
        \|f\|_{L^2 (M)}^r&\lesssim A(\lambda, q, R)^r\|\mathds{1}_{[\lambda, \lambda+1]} (P)f\|_{L^q (\mathcal{T}_{1/R}(\Sigma))}^{r-1}\cdot \left(\frac{1}{R} \right)^{\frac{n-k}{q}} \lambda^{\delta(k, q)} (\# S)\|f\|_{L^q (M)},
    \end{split}
\end{align*}
that is, for $2\leq q\leq \infty$ and $1<r<\infty$,
\begin{align}\label{eq:upp-bd-for-f-tnbhd}
    \frac{\|f\|_{L^q (M)}}{\|\mathds{1}_{[\lambda, \lambda+1]}(P)f\|_{L^q (\mathcal{T}_{1/R}(\Sigma))}}\lesssim A(\lambda, q, R)^{\frac{r}{r-1}}(\# S)^{\frac{1}{r-1}} R^{-\frac{n-k}{q(r-1)}} \lambda^{\frac{\delta(k, q)}{r-1}}.
\end{align}
Combining \eqref{eq:lower-bd-for-f-tnbhd} and \eqref{eq:upp-bd-for-f-tnbhd}, we have
\begin{align}\label{eq:lo-up-bds}
    \frac{R^{\frac{n-k}{q}}}{\left(\sup_{|y|\leq C/R} |\Sigma_y|^{\frac{1}{p'q}} \right)A(\lambda, q, R)^{\frac{r}{r-1}}\lambda^{\delta(k, pq)}}\lesssim (\# S)^{\frac{1}{r-1}} R^{-\frac{n-k}{q(r-1)}} \lambda^{\frac{\delta(k, q)}{r-1}},\quad \text{for } 1\leq p \leq \infty.
\end{align}
Since $\Sigma_y$ is a smooth perturbation of $\Sigma$ and $R\gg 1$ is sufficiently large, we have
\begin{align*}
    |\mathcal{T}_{1/R} (\Sigma)|\approx (1/R)^{n-k}\left(\sup_{|y|\leq C/R} |\Sigma_y|\right).
\end{align*}
By this and \eqref{eq:lo-up-bds}, we have
\begin{align*}
    (\# S)^{\frac{1}{r-1}}|\mathcal{T}_{1/R}(\Sigma)|^{\frac{1}{p'q}}\gtrsim \left(\frac{1}{R} \right)^{\frac{n-k}{p'q}} \cdot \frac{R^{(n-k)\left(\frac{1}{q}+\frac{1}{q(r-1)} \right)}}{A(\lambda, q, R)^{\frac{r}{r-1}}\lambda^{\delta(k, pq)+\frac{\delta(k, q)}{r-1}}},\quad 1\leq p\leq \infty.
\end{align*}
This proves \eqref{eq:up-est-for-tnbhd} when $(k, pq)\not=(n-2, 2)\not=(k, q)$. The other cases follow similar if we apply the other cases in \eqref{eq:bgt-hu-est}. This completes the proof of Proposition \ref{prop:up-tubular-nbhd}.

\begin{remark}[Possible sharpness on tori with $\lambda^\epsilon$-loss]\label{remark:sharpness}
    For $\lambda^{-1}\leq \epsilon(\lambda)\leq 1$, suppose we have the following analogue of \eqref{eq:bgt-hu-est}:
    \begin{align}\label{eq:stability-in-remark}
        \|\mathds{1}_{[\lambda-\epsilon(\lambda), \lambda+\epsilon(\lambda)]}(P) \|_{L^2 (M)\to L^q (\Sigma_y)}\leq \tilde{B}(\lambda, k, q),
    \end{align}
    that is, the bounds of $\|\mathds{1}_{[\lambda-\epsilon(\lambda), \lambda+\epsilon(\lambda)]}(P) \|_{L^2 (M)\to L^q (\Sigma_y)}$ are stable under small smooth perturbations of the submanifold $\Sigma$. By the proof of Proposition \ref{prop:up-tubular-nbhd}, we have an analogue of \eqref{eq:up-est-for-tnbhd} as follows:
    \begin{align}\label{eq:general-up-assuming-stability}
        (\# S_1)^{\frac{1}{2}}|\mathcal{T}_{1/R}(\Sigma)|^{\frac{1}{2}}\geq \frac{CR^{\frac{n-k}{4}}}{A(\lambda, 2, R)^{\frac{3}{2}}\tilde{B}(\lambda, k, \infty)\tilde{B}(\lambda, k, 2)^{\frac{1}{2}}},
    \end{align}
    where $S_1=\{\lambda\in [\lambda-\epsilon(\lambda), \lambda+\epsilon(\lambda)]:\lambda\in \mathrm{spec}(P)\}$. This estimate \eqref{eq:general-up-assuming-stability} corresponds to the case $(p, q, r)=(\infty, 2, 3)$ in Proposition \ref{prop:up-tubular-nbhd}. We now consider the flat torus $\mathbb{T}^2$. It is known that $\# S_1$ is at most $1$ when $\epsilon(\lambda)=\lambda^{-1}$ on tori. Suppose $\#S_1=1, \epsilon(\lambda)=\lambda^{-1}$, $-\Delta_{\mathbb{T}^2}e_\lambda=\lambda^2 e_\lambda$, and $\Sigma$ is a smooth curve with nonzero curvature. By \cite[Main Theorem]{BourgainRudnick2012RestrictionNodal}, there exists a $\Lambda>0$ such that if $\lambda\geq \Lambda$, then
    \begin{align}\label{eq:br-curve-est}
        c_\Sigma \|e_\lambda\|_{L^2 (\mathbb{T}^2)}\leq \|e_\lambda\|_{L^2 (\Sigma)}\leq C_\Sigma \|e_\lambda\|_{L^2 (\mathbb{T}^2)},\quad \text{for some } c_\Sigma, C_\Sigma>0.
    \end{align}
    Here, $c_\Sigma, C_\Sigma$ depend on $\Sigma$. Now suppose that for a sufficiently large but fixed $R\gg 1$, there is a $\tilde{\Lambda}>0$ such that if $\lambda\geq \tilde{\Lambda}$, then there are constant $c_1, C_1>0$ satisfying
    \begin{align}\label{eq:br-analogue}
        c_1\|e_\lambda\|_{L^2 (\mathbb{T}^2)}\leq \|e_\lambda\|_{L^2 (\Sigma_y)}\leq C_1 \|e_\lambda\|_{L^2 (\mathbb{T}^2)},\quad \text{for all } |y|\leq 1/R,
    \end{align}
    where $c_1, C_1$ are independent of $\Sigma$. This might be a reasonable assumption, since $c_\Sigma, C_\Sigma$ may be uniform under small smooth perturbations of $\Sigma$ in the proof for the $\mathbb{T}^2$ case in \cite[Introduction]{BourgainRudnick2012RestrictionNodal}. By \eqref{eq:br-analogue}, we can take $\tilde{B} (\lambda, 1, 2)=1$, where the implicit constant is absorbed in $C$ in \eqref{eq:general-up-assuming-stability}. Also, by \eqref{eq:br-analogue}, we have
    \begin{align*}
        \|e_\lambda\|_{L^2 (\mathcal{T}_{1/R}(\Sigma))}^2=\int_{|y|\leq 1/R} \|e_\lambda\|_{L^2 (\Sigma_y)}^2\:dy \geq c_1^2\int_{|y|\leq 1/R} \|e_\lambda\|_{L^2 (\mathbb{T}^2)}^2\:dy=\frac{c_1^2}{R}\|e_\lambda\|_{L^2 (\mathbb{T}^2)}^2,
    \end{align*}
    since $\Sigma_y$ in $\mathbb{T}^2$ is just a parallel translation of $\Sigma$. Thus, we can take $A(\lambda, 2, R)=\sqrt{R}$, where $c_1$ is absorbed in $C$ in \eqref{eq:general-up-assuming-stability}. Moreover, by \cite[(1.10)]{BurqGerardTzvetkov2007restrictions}, we know that
    \begin{align*}
        \|e_\lambda\|_{L^\infty (\Sigma_y)}\leq \|e_\lambda\|_{L^\infty (\mathbb{T}^2)}\lesssim \lambda^\epsilon \|e_\lambda\|_{L^2 (\mathbb{T}^2)},
    \end{align*}
    and so, we can take $\tilde{B}(\lambda, 1, \infty)=\lambda^\epsilon$, where the implicit constant is independent of $\Sigma$ and can be absorbed into $C$ in \eqref{eq:general-up-assuming-stability}.

    Putting these together, the left and right sides of \eqref{eq:general-up-assuming-stability} are, for $(n, k)=(2, 1)$ in our case,
    \begin{align*}
        (\# S_1)^{\frac{1}{2}}|\mathcal{T}_{1/R}(\Sigma)|^{\frac{1}{2}}=\frac{2}{R^{1/2}},\quad \frac{C R^{\frac{1}{4}}}{A(\lambda, 2, R)^{\frac{3}{2}}\tilde{B}(\lambda, 1, \infty)\tilde{B}(\lambda, 1, 2)^{\frac{1}{2}}}=\frac{C R^{1/4}}{R^{\frac{1}{2}\cdot \frac{3}{2}}\cdot \lambda^\epsilon\cdot 1}=\frac{C}{\lambda^\epsilon R^{\frac{1}{2}}}.
    \end{align*}
    Hence, \eqref{eq:general-up-assuming-stability} may be sharp up to some constant $C>0$ with a $\lambda^\epsilon$-loss on tori where $\Sigma$ is a smooth curve with nonzero geodesic curvature.
\end{remark}

\section{Proof of Corollary \ref{cor:up-tubular-nbhd}}\label{sec:pf-of-cor-up-tnbhd}

To obtain Corollary \ref{cor:up-tubular-nbhd}, one can follow the same argument of the proof of Proposition \ref{prop:up-tubular-nbhd}. The difference is that $[\lambda, \lambda+1]$ is replaced by $[0, \lambda_0]$. With that in mind, instead of using \eqref{eq:1s0p-tnbhd-short-comp}, we want to use
\begin{align*}
    \begin{split}
        \|\mathds{1}_{[0, \lambda_0]}(P)f\|_{L^q (\mathcal{T}_{1/R} (\Sigma))}^q&\lesssim \left(\sup_{|y|\leq C/R} |\Sigma_y|^{\frac{1}{p'}} \right)\int_{|y|\leq C/R} \left(\sum_{l=1}^{\lambda_0} l^{\delta(k, pq)} \right)^q\|f\|_{L^2 (M)}^q\:dy \\
        &\lesssim \left(\sup_{|y|\leq C/R} |\Sigma_y|^{\frac{1}{p'}} \right) \left(\frac{1}{R} \right)^{n-k} \lambda_0^{q\delta(k, pq)+q}\|f\|_{L^2 (M)}^q,
    \end{split}
\end{align*}
and instead of \eqref{eq:1st-comp-for-upbd-f-tnbhd}, we want to use
\begin{align*}
    \begin{split}
        &\|\mathds{1}_{[0, \lambda_0]}(P)f\|_{L^q (\mathcal{T}_{1/R}(\Sigma))}\lesssim \left(\int_{|y|\leq C/R} \|\mathds{1}_{[0, \lambda_0]}(P) f\|_{L^q (\Sigma_y)}^q\:dy \right)^{\frac{1}{q}} \\
        &\lesssim \left(\frac{1}{R} \right)^{\frac{n-k}{q}} \left(\sum_{l=1}^{\lambda_0} l^{\delta(k, q)} \right) \|\mathds{1}_{[0, \lambda_0]}(P)f\|_{L^2 (M)} \lesssim \left(\frac{1}{R} \right)^{\frac{n-k}{q}} \lambda_0^{\delta(k, q)+1} (\# S_0)\|f\|_{L^2 (M)}.
    \end{split}
\end{align*}
All the other computations are similar.

\bibliographystyle{amsplain}
\bibliography{references}

\end{document}